\documentclass[12pt]{amsart}
\usepackage{amsmath,amssymb,mathscinet}

\title[Rationality of Fano threefolds]{Rationality of Fano threefolds of degree $18$ over nonclosed fields}
\author{Brendan Hassett and Yuri Tschinkel}
\date{October 29, 2019}

\newcommand{\bC}{\mathbb C}
\newcommand{\bF}{\mathbb F}

\newcommand{\bP}{\mathbb P}
\newcommand{\bZ}{\mathbb Z}

\newcommand{\cD}{\mathcal D}

\newcommand{\cH}{\mathcal{H}}

\newcommand{\cJ}{\mathcal J}

\newcommand{\cO}{\mathcal O}

\newcommand{\cX}{\mathcal X}

\newcommand{\ra}{\rightarrow}

\newcommand{\Bl}{\operatorname{Bl}}
\newcommand{\Br}{\operatorname{Br}}

\newcommand{\Pic}{\operatorname{Pic}}

\newcommand{\spa}{\operatorname{span}}

\theoremstyle{plain}

\theoremstyle{plain}
\newtheorem{prop}{Proposition}

\newtheorem{theo}[prop]{Theorem}
\newtheorem{coro}[prop]{Corollary}

\theoremstyle{definition}

\newtheorem{ques}[prop]{Question}

\begin{document}

\maketitle

\section{Introduction}
Manin \cite{Manin3fold} proposed to study (uni)rationality
of Fano threefolds over nonclosed fields, in situations where geometric
(uni)rationality is known.
In cases where the Picard group is generated
by the canonical class, i.e., those of rank and index one, he assigned an
`Exercise' \cite[p.~47]{Manin3fold} to explore the rationality of 
degree $12, 16, 18$, and $22$. See \cite[p.~215]{IskPro} for a list of
geometrically rational Fano threefolds of rank one. 

We have effective criteria for deciding the rationality of surfaces 
over nonclosed fields -- the key invariant is the Galois action on the
geometric Picard group. This invariant is trivial for Fano threefolds
considered above.  

Kuznetsov and Prokhorov are pursuing the rationality 
question from the perspective of
derived categories, relating the derived category of the threefold to 
categories of twisted sheaves on auxillary curves. Here we take a 
more geometric approach based on the analysis of torsors over the
intermediate Jacobians presented in \cite{HTcycle,BW}. Throughout, we work over
a field $k$ of characteristic zero. Our main result is:
\begin{theo} \label{theo:main}
Let $X$ be a smooth Fano threefold of degree $18$ defined over $k$ and
admitting a $k$-rational point. Then $X$ is rational over $k$ if and only if
$X$ admits a 
conic over $k$.
\end{theo}
Here a conic means a geometrically connected curve of degree two -- possibly
nonreduced or reducible.

Much recent work on rationality has focused on applications of specialization
techniques to show the {\em failure} of (stable) rationality
\cite{VoisinInv,CTP,HTK,Totaro,HPTActa,Schreieder,NS,KTspecialize}. Here we use it to 
{\em prove} rationality, avoiding complicated case-by-case arguments for special 
geometric configurations; see Theorem~\ref{theo:speciality}. This technique
was also used to analyze rationality for cubic fourfolds \cite{RS}.

\

{\bf Acknowledgments:} The first author was partially supported by NSF grant
1701659 and the Simons Foundation.

\section{Projection constructions}
Let $X$ be a smooth Fano threefold of degree $18$ over $k$.

\subsection{Projection from lines}
The variety of lines $R_1(X)$ is nonempty and connected of pure 
dimension one \cite[Prop.~4.2.2]{IskPro}
and sweeps out a divisor in $X$ with class $-3K_X$ \cite[Th.~4.2.7]{IskPro}.
For generic $X$, $R_1(X)$ is a smooth curve of genus 
ten \cite[Th.4.2.7]{IskPro}.
If $R_1(X)$ is smooth then $X$ admits no nonreduced conics 
\cite[Rem.~2.1.7]{KuzPro}.

Suppose that $\ell \subset X$ is a line and write $\widetilde{X}$
for the blow-up of $X$ along $\ell$. Then double projection along
$\ell$ induces a birational map
$$\widetilde{X} \dashrightarrow {\widetilde{X}}^+ \ra Y,$$
where $Y\subset \bP^4$ is a smooth quadric hypersurface \cite[Th.~4.3.3]{IskPro}.
This flops the three lines incident to $\ell$ and contracts a divisor
$$D\in |-2K_{{\widetilde{X}}^+} - 3 E^+|$$
to a smooth curve $C\subset Y$ of degree seven and genus two.

Since $Y$ admits a $k$-rational point it is rational over $k$; the
same holds true for $X$.

\begin{prop}
If $X$ is a Fano threefold of degree $18$ admitting a line over $k$
then $X$ is rational.
\end{prop}

\subsection{Projection from conics}
\label{subsect:conic}
We discuss the structure of the variety $R_2(X)$ of conics on $X$:
\begin{itemize}
\item{$R_2(X)$ is nonempty of pure dimension two \cite[Th.~4.5.10]{IskPro}.}
\item{$R_2(X)$ is geometrically isomorphic to the Jacobian of a 
genus two curve $C$ \cite[Prop.~3]{IlMa} \cite[Th.~1.1.1]{KuzPro}.}
\item{
Through each point of $X$ there pass finitely many conics
\cite[Lem.~4.2.6]{IskPro}; indeed, through a generic such point
we have nine conics \cite[2.8.1]{Takeuchi}.}
\item{Given a conic $D \subset X$, let $\widetilde{X}$ denote the
blowup of $X$ along $D$. Then double projection along $D$ induces
a fibration \cite[Cor.~4.4.3,Th.~4.4.11]{IskPro}
$$X \dashrightarrow \widetilde{X}^+ \stackrel{\phi}{\ra} \bP^2$$
in conics with quartic degeneracy curve.}
\end{itemize}

\subsection{Projection from points}
\label{subsect:projpoint}

We recall the results of Takeuchi \cite{Takeuchi} presented in
\cite[Th.~4.5.8]{IskPro}.
Let $\widetilde{X}$ denote the blowup of $X$ at $x$,
with exceptional divisor $E$.
\begin{prop} \label{prop:Takeuchi}
Suppose we have a point $x\in X(k)$ and let $\widetilde{X}$ denote
the blowup of $X$ at $x$. We assume that
\begin{itemize}
\item{$x$ does not lie on a line in $X$;}
\item{there are no effective divisors $D$ on $X$ such that
$(K_{\widetilde{X}}^2) \cdot D =0$.}
\end{itemize}
Then triple-projection from $x$ gives a fibration
$$X \stackrel{\sim}{\dashrightarrow} {\widetilde{X}}^+ \stackrel{\phi}{\ra} \bP^1$$
in sextic del Pezzo surfaces.
\end{prop}

We offer a more detailed analysis of double projection from a point 
$x\in X(k)$ not on a line. 
By \cite[\S~4.5]{IskPro} 
the projection morphism
$$\tilde{\phi}: \widetilde{X} \ra \bP^7$$
is generically finite onto its image $\overline{X}$
and the Stein factorization
$$ \widetilde{X} \stackrel{\phi'}{\ra} X' 
\stackrel{\overline{\phi}}{\ra} \overline{X}$$
yields a Fano threefold of genus six with canonical Gorenstein singularities.
The condition precluding effective divisors $D$
with $(K_{\widetilde{X}})^2 D=0$ means that $\overline{\phi}$ admits no
exceptional divisors.
The nontrivial fibers of $\phi'$ are all isomorphic to $\bP^1$'s, with
the following possible images in $X$:
\begin{enumerate}
\item{a conic in $X$ through $x$;}
\item{a quartic curve of arithmetic genus one
in $X$, spanning a $\bP^3$, with a singularity of multiplicity two at $x$;}
\item{a sextic curve of arithmetic genus two 
in $X$, spanning a $\bP^4$, with a singularity of multiplicity three at $x$.}
\end{enumerate}
Moreover, if $\phi'$ does not contract any surfaces then the exceptional
divisor $E$ over $x$ is embedded
in $\bP^7$ as a Veronese surface.

The quartic curves on $X$ with node at a fixed point
$x$ have expected dimension $0$. The sextic curves on $X$ with
transverse triple point at a fixed point $x$ have expected dimension
$-1$. Indeed, we have:
\begin{prop} \label{prop:howgeneric}
\cite[Prop.~4.5.1]{IskPro}
Retain the notation above.
For a generic $x\in X$ 
\begin{itemize}
\item{the quartic and sextic curves described above do not occur;}
\item{$\phi'$ is a small contraction;}
\item{the rational map $X \stackrel{\sim}{\dashrightarrow} {\widetilde{X}}^+$
factors as follows
\begin{enumerate}
\item{blow up the point $x$;}
\item{flop the nine conics through $x$;}
\end{enumerate}
}
\item{$\phi$ restricts to the proper transform $E^+$ of $E$
as an elliptic fibration associated with cubics 
based at nine points.}
\end{itemize}
\end{prop}

\section{Unirationality constructions}

In this section, we consider the following question inspired by
\cite[p.~46]{Manin3fold}:
\begin{ques} 
Let $X$ be a Fano threefold of degree $18$ over $k$.
Suppose $X(k) \neq \emptyset$. Is $X$ is unirational over $k$?
\end{ques}
From our perspective, unirationality is more delicate than rationality 
as we lack a specialization theorem for smooth families in this context.
We cannot apply the theorem of \cite{KTspecialize} -- as we do 
in the proof of Theorem~\ref{theo:speciality} -- to reduce to 
configurations in general position. 

The geometric constructions below highlight some of the issues that arise.

\subsection{Using a point}
\begin{prop} \label{prop:unirat3}
Let $X$ be a Fano threefold of degree $18$ over $k$ admitting
a point $x\in X(k)$ satisfying the condition in Proposition~\ref{prop:Takeuchi}.
Then $X$ is unirational over $k$ and rational points are Zariski dense.
\end{prop}
\begin{proof}
We retain the notation from Proposition~\ref{prop:Takeuchi}.
Note that the proper transforms of lines $L \subset E^+$ give trisections
of our del Pezzo fibration
$$\phi: {\widetilde{X}}^+ \ra \bP^1.$$
Basechanging to $L$ yields
$$\phi_L: {\widetilde{X}}^+ \times_{\bP^1}L \ra L,$$
a fibration of sextic del Pezzo surfaces with a section.
Thus the generic fiber of $\phi_L$ is rational over $k(L)$ by
\cite[p.~77]{Manin66}.
Since $L\simeq \bP^1$, the total space of the fibration is rational
over $k$. As it dominates ${\widetilde{X}}^+$, we conclude that
$X$ is unirational.
\end{proof}

If the rational points are Zariski dense then we can find one
where Proposition~\ref{prop:Takeuchi} applies.
However, if we are given only a single rational point on $X$
we must make
a complete analysis of degenerate cases as partly described in
Section~\ref{subsect:projpoint}. In addition, we must consider cases
where there exist lines over $\bar{k}$
passing through our given rational point.

For instance, consider the case where a single line
$x\in \ell \subset X$. To resolve the double projection
at $x$, we must take the following steps:
\begin{itemize}
\item{blow up $x$ to obtain an exceptional divisor $E_1\simeq \bP^2$;}
\item{blow up the proper transform $\ell'$ of the line $\ell$ 
with 
$$N_{\ell'}= \cO_{\bP^1}(-1) \oplus \cO_{\bP^1}(-2)$$
to obtain an exceptional divisor $E_2 \simeq \bF_1$.}
\end{itemize}
Let $E_1'$ and denote the proper transform of $E_1$ in the
second blowups.
The linear series resolving the double
projection is 
$$h - 2E'_1 - E_2$$
which takes $E'_1 \simeq \bF_1$ to a cubic scroll, 
$E_2$ to a copy of $\bP^2$, and the $(-1)$-curve on $E_2$ to an
ordinary singularity on the image. 
The induced contraction
$$\phi': \widetilde{X} \ra X' \subset \bP^7$$
has degree
$$(h - 2E'_1 - E_2)^3 = 10.$$
Thus $X'$ admits a `degenerate Veronese surface' consisting of a 
cubic scroll and a plane meeting along a line
coinciding with the $(-1)$-curve of the scroll; $X'$ has an ordinary
singularity along that line.  

Of course, the most relevant degenerate cases for arithmetic purposes 
involve multiple lines through $x$ conjugated over the ground field. 
It would be interesting to characterize the possibilities.

\subsection{Using a point and a conic}
Here is another approach:
Let $X$ admit
a point $x\in X(k)$ and a conic $D\subset X$ defined over $k$.
The results recalled in Section~\ref{subsect:conic} imply that
$X$ is birational over $k$ to 
$$\phi: \widetilde{X}^+ \ra \bP^2,$$
a conic bundle degenerating over a plane quartic curve $B$. 

Suppose there exists a rational point on $\widetilde{X}^+$
whose image $p\in \bP^2$ is not contained in the degeneracy curve.
Consider the pencil of lines through $p$. The corresponding
pencil of surfaces on $\widetilde{X}^+$ are conic bundles
over $\bP^1$ with four degenerate fibers and the resulting
fibration admits a section. Such a surface is either
isomorphic to a quartic del Pezzo surface or birational to such 
a surface \cite[p.~48]{KST}. It is a classical fact that a quartic
del Pezzo surface with a rational point is unirational. This yields the unirationality of $X$ over $k$.

The argument works even when $p$ is a smooth point or node
of $B$. Here 
we necessarily have higher-order ramification over the nodes -- this is because
the associated generalized Prym variety is compact -- which
we can use to produce a section of the resulting pencil
of degenerate quartic del Pezzo surfaces. However, there is
trouble when $p$ is a cusp of $B$.

\section{Rationality results}
Our first statement describes the rationality construction under favorable
genericity assumptions:
\begin{prop} \label{prop:genericityrat}
Let $X$ be a Fano threefold of degree $18$ over $k$. Assume that:
\begin{itemize}
\item{there exists an $x\in X(k)$ satisfying the conditions of 
Proposition~\ref{prop:Takeuchi} so that $X$ is birational to a fibration
$\phi:{\widetilde{X}}^+ \ra \bP^1$
in sextic del Pezzo surfaces;}
\item{there exists an irreducible curve $M\subset X$, disjoint from the indeterminacy of
$X\stackrel{\sim}{\dashrightarrow} {\widetilde{X}}^+$, with degree prime to three.}
\end{itemize}
Then $X$ is rational over $k$.
\end{prop}
\begin{proof}
We saw in the proof of 
Proposition~\ref{prop:unirat3} that the generic fiber $S$ of $\phi$ is a 
sextic del Pezzo surface admitting a rational point of a degree-three
extension. Our assumptions imply that
$S\cdot M = \deg(M)$ which is prime to three, so applying \cite[p.~77]{Manin66}
we conclude that $S$ is rational over $k(\bP^1)$ and $X$ is rational 
over $k$.
\end{proof} 

We now show these genericity assumptions are not necessary:
\begin{theo} \label{theo:speciality}
Let $X$ be a Fano threefold of degree $18$ over $k$. Assume that $X$ admits a 
rational point $x$ 
and a conic $D$, both defined over $k$. Then $X$ is rational.
\end{theo}
\begin{proof}
Let $B$ denote the Hilbert scheme of all triples
$(X,x,D)$ of objects described in the statement. This is smooth and connected
over the moduli stack of degree $18$ Fano threefolds; indeed, we saw in
Section~\ref{subsect:conic} that the
parameter space of conics on $X$ is an abelian surface. The moduli
stack itself is a smooth Deligne-Mumford stack since
Kodaira vanishing gives $H^i(T_X)=0$ for $i=2,3$ and $H^0(T_X)=0$ by
\cite{ProkhorovAut}. The classification of Fano threefolds shows that
the moduli stack is connected. Thus $B$ is smooth and connected.

Consider the 
universal family
$$(\cX\stackrel{\pi}{\ra} B, \mathbf{x}:B\ra \cX, \cD\subset \cX), $$
where $\pi$ is smooth and projective. The generic fiber of
$\pi$ is rational over $k(B)$ 
as the genericity conditions of Proposition~\ref{prop:genericityrat} are
tautologically satisfied -- see Proposition~\ref{prop:howgeneric} for details.
The specialization theorem \cite[Th.~1]{KTspecialize} implies that every
$k$-fiber of $\pi$ is rational over $k$. This theorem assumes that the base
is a curve. However, our parameter space $B$ is smooth so Bertini's Theorem 
implies that each $b\in B(k)$ may be connected to the generic point by a curve
smooth at $b$.  
\end{proof}

\section{Analysis of principal homogeneous spaces}
\subsection{Proof of Theorem~\ref{theo:main}}
One direction is Theorem~\ref{theo:speciality}; we focus on the
converse.
Suppose that $X$ is rational over the ground field. 
Let $C$ be the genus two curve whose Jacobian $J(C)$ is isomorphic
to the intermediate Jacobian $IJ(X)$ over $k$. 
The mechanism of \cite[\S~5]{HTcycle} gives a principal homogeneous space
$P$ over $J(C)$ with the property that the Hilbert scheme $\cH_d$
parametrizing
irreducible curves of degree $d$ admits a morphism 
$$\cH_d \ra P_d$$
descending the Abel-Jacobi map to $k$,
where $[P_d]=d[P]$ in the Weil-Ch\^atelet group of $J(C)$. 
By Theorem~22 of \cite{HTcycle}, if $X$ is rational then 
$P \simeq \Pic^i(C)$ for $i=0$ or $1$. 
In particular, we have 
$$R_1(X) \hookrightarrow P$$
and by the known results of Section~\ref{subsect:conic}
$$R_2(X) \simeq P_2 \simeq J(C).$$
Indeed, since $C$ has genus two we have identifications
$$J(C) = \Pic^0(C) \simeq \Pic^2(C),$$
which gives the desired interpretation of $P_2$ whether
$P=\Pic^0(C)$ or $\Pic^1(C)$. As a consequence, $R_2(X)$
admits a $k$-rational point. 

\subsection{A corollary to Theorem~\ref{theo:main}}
Retain the notation of the previous section.
Without assumptions on the existence of points or conics
on $X$ defined over $k$, we know that 
$$18[P]=0 \text{ and } 9[R_2(X)]=0$$
in the Weil-Ch\^atelet group.
This allows us to deduce an extension of our main result:
\begin{coro}
Let $X$ be Fano threefold of degree of degree $18$ over $k$
with $X(k)\neq \emptyset$. Suppose that $X$ admits a curve of
degree prime to three, defined over $k$. Then $X$ is rational.
\end{coro}
Our assumption means that $2[P]=0$, whence $[R_2(X)]=0$ and
$X$ admits a conic defined over $k$. Hence Theorem~\ref{theo:main}
applies.

\subsection{Generic behavior}
There are examples over function fields where the
principal homogeneous space is not annihilated
by two:
\begin{prop}
Over $k=\bC(\bP^2)$, there exist examples of $X$ such that
the order of $[P]$ is divisible by three.
\end{prop}

\begin{proof}
Let $S$ be a complex K3 surface with $\Pic(S)=\bZ h$, with
$h^2=18$. Mukai \cite{Mukai} has shown that $S$
is a codimension-three linear section of a homogeneous
space $W \subset \bP^{13}$ arising as the closed orbit
for the adjoint representation of $G_2$
$$S = \bP^{10} \cap  W.$$
Consider the associated net of Fano threefolds
$$
\varpi:\cX \ra \bP^2,
$$
obtained by intersecting $W$ with codimension-two linear
subspaces 
$$\bP^{10} \subset \bP^{11} \subset \bP^{13}.$$
Write $X$ for the generic fiber over $\bC(\bP^2)$.

Let $R_2(\cX/\bP^2)$ denote the relative variety of conics. 
This was analyzed in \cite[\S~3.1]{IlMa}: 
The conics in fibers of $\varpi$ cut out pairs of points on $S$,
yielding a birational identification and natural abelian fibration
$$S^{[2]} \stackrel{\sim}{\dashrightarrow} R_2(\cX/\bP^2) 
\stackrel{\psi}{\ra} \bP^2.$$
The corresponding principal homogeneous space has order 
divisible by nine; its order is divisible by three if it is
nontrivial.

These fibrations are analyzed in more depth in \cite[\S~3.3]{MSTVA}
and \cite{KR}.
Let $T$ denote the moduli space of rank-three stable 
vector bundles $V$ on $S$ with $c_1(V)=h$ and $\chi(V)=6$.
Then we have
\begin{itemize}
\item{$T$ is a K3 surface of degree two;}
\item{the primitive cohomology of $S$ arises as an index-three
sublattice of the primitive cohomology of $T$
$$H^2(S,\bZ)_{prim} \subset H^2(T,\bZ)_{prim},$$
compatibly with Hodge structures;}
\item{the Hilbert scheme $T^{[2]}$ is birational to
the relative Jacobian
fibration of the degree-two linear series on $T$
$$\cJ \ra \bP^2;$$
}
\item{the relative Jacobian fibration of $\psi$
is birational to $\cJ$ over $\bP^2$.}
\end{itemize}
The last statement follows from \cite[p.~486]{Sawon} or
\cite[\S~4]{Mark}: The abelian fibration $\psi$
is realized as a twist of the 
fibration $\cJ \ra \bP^2$; the twisting data is
encoded by an element $\alpha \in \Br(T)[3]$ annihilating
$H^2(S,\bZ)_{prim}$ modulo three.  

Now suppose that $\psi$ had a section. Then $\cJ$ and 
$S^{[2]}$ would be birational holomorphic symplectic varieties.
The Torelli Theorem implies that
their transcendental degree-two cohomology
-- $H^2(T,\bZ)_{prim}$ and $H^2(S,\bZ)_{prim}$ respectively --
are isomorphic. This contradicts our computation above.
\end{proof}

\subsection{Connections with complete intersections?}
Assume $k$ is algebraically closed and $X$ a Fano threefold of degree
$18$ over $k$. Kuznetsov, Prokhorov, and Shramov \cite{KuzPro} have
pointed out the existence of a smooth complete intersection of
two quadrics $Y\subset \bP^5$ with
\begin{equation} \label{eqn:XYiso}
R_1(Y) \simeq R_2(X),
\end{equation}
Both have intermediate Jacobian isomorphic to the Jacobian
of a genus two curve $C$.

Now suppose that $X$ and $Y$ are defined over a nonclosed field $k$
with $IJ(X)\simeq IJ(Y)$. In general, we would not expect
$R_2(X)$ and $R_1(Y)$ to be related as principal homogeneous
spaces; for example, we generally have
$9[R_2(X)]=0$ and $4[R_1(Y)]=0$ (see \cite{HT2quad}).

Verra \cite{VerraSlide} has found a direct connection between
complete intersections of quadrics and {\em singular} Fano threefolds
of degree $18$. Suppose we have a twisted cubic curve
$$R \subset Y \subset \bP^5,$$
which forces $Y$ to be rational.
Consider the linear series of quadrics vanishing along $R$;
the resulting morphism
$$\Bl_R(Y) \ra \bP^{11}$$
collapses the line residual to $R$ in $\spa(R)\cap Y$.
Its image $X_0$ is a nodal Fano threefold of degree $18$.

\bibliographystyle{alpha}
\bibliography{degree18}
\end{document}